\documentclass[dvipsnames,12pt]{amsart}
\usepackage[utf8]{inputenc}
\usepackage[margin=1.25in,marginparwidth=1in]{geometry}
\usepackage{mathrsfs}

\usepackage{booktabs} 
\usepackage{setspace} 
\usepackage{cancel}   
\usepackage{enumitem} 
\usepackage{lipsum}   
\usepackage{etoolbox} 
\usepackage[margin=2cm]{caption} 
\usepackage{multirow} 
\usepackage{mathtools}

\usepackage{bbm}
\usepackage{amssymb,amsmath,amsthm}
\usepackage[notref, notcite, final]{showkeys}
\usepackage{todonotes}

\allowdisplaybreaks 

\usepackage[colorlinks, bookmarksnumbered, bookmarks]{hyperref} 

\usepackage{xcolor}

\usepackage{comment}

\theoremstyle{definition} 
\newtheorem{theorem}{Theorem}[section] 
\newtheorem{def_temp}[theorem]{Definition}
\newtheorem{question}{Question}

\newtheorem{lemma}[theorem]{Lemma}
\newtheorem{proposition}[theorem]{Proposition}
\newtheorem{corollary}[theorem]{Corollary}

\newtheorem{eg_temp}[theorem]{Example}
\newtheorem{rmk_temp}[theorem]{Remark}
\numberwithin{equation}{section} 

  \newtheorem{claim}{Claim}[theorem]
  
\newenvironment{claimproof}[1][Proof of Claim]{\noindent \underline{#1.} }{\hfill$\diamondsuit$}

\newenvironment{remark}
  {\pushQED{\qed}\begin{rmk_temp}}
  {\popQED\end{rmk_temp}}
\newenvironment{example}
  {\pushQED{\qed}\begin{eg_temp}}
  {\popQED\end{eg_temp}}
\newenvironment{definition}
  {\pushQED{\qed}\begin{def_temp}}
  {\popQED\end{def_temp}}

\definecolor{maroon}{rgb}{0.35,0,0}

\newcommand{\N}{\mathbb{N}}
\newcommand{\Z}{\mathbb{Z}}
\newcommand{\R}{\mathbb{R}}
\newcommand{\C}{\mathbb{C}}

\newcommand{\eps}{{\varepsilon}}

\newcommand{\vect}[1]{\boldsymbol #1}
\newcommand{\vectx}{{\vect x}}

\newcommand{\sltwor}{\mathsf{SL}(2,\R)}
\DeclareMathOperator{\trace}{trace}
\DeclareMathOperator{\Var}{Var}
\DeclareMathOperator{\spectrum}{spec}

\DeclareMathOperator{\lcm}{lcm}
\newcommand{\polynomial}{{\operatorname{poly}}}

\newcommand{\scrK}{{\mathscr{K}}}

\let\oldsqrt\sqrt
\def\sqrt{\mathpalette\DHLhksqrt}
\def\DHLhksqrt#1#2{%
\setbox0=\hbox{$#1\oldsqrt{#2\,}$}\dimen0=\ht0
\advance\dimen0-0.2\ht0
\setbox2=\hbox{\vrule height\ht0 depth -\dimen0}%
{\box0\lower0.4pt\box2}}

\AtBeginDocument{

}

\title[Gap-Rich Sets]{On the Identification of Certain Classes of \\ Gap-Rich Sets}

\author[J. Fillman]{Jake Fillman} 
\author[Y.-S. Lim]{Yi-Sheng Lim}

\address{ Department of Mathematics, Texas A\&M University, College Station, TX 77843}
\email{\href{mailto:fillman@tamu.edu}{fillman@tamu.edu}}

\email{\href{mailto:yishenglimysl@tamu.edu}{yishenglimysl@tamu.edu}}

\date{}

\begin{document}

\begin{abstract} 
The notion of gap-richness was recently introduced to quantify the ability to open spectral gaps for periodic word models by arbitrarily small perturbations and is useful in the construction of almost-periodic operators having exotic spectral properties, such as Cantor spectrum of zero Hausdorff dimension.
Given a subset of Euclidean space that is locally dense in the Zariski topology, we show that any image of such a set under a nonconstant polynomial is gap-rich.
This generalizes the scope of previous work in several ways, including allowing local correlations into potential blocks and allowing for totally disconnected sets of realizations.
\end{abstract}

\maketitle

 \vskip 0.5cm
    

    \vskip 0.5cm


\onehalfspacing

\hypersetup{
	linkcolor={black!30!blue},
	citecolor={black!10!red},
	urlcolor={black!30!green}
}

\section{Introduction}\label{sec:intro}

This work is about spectral gaps for discrete periodic Schr\"odinger operators.
To set the stage, we recall that a
discrete Schr\"odinger operator $H_V:\ell^2(\Z) \to \ell^2(\Z)$ with (bounded) potential $V:\Z \to \R$
is given by
\begin{equation}
    [H_V\psi]_n
    = \psi_{n-1} + \psi_{n+1} + V_n \psi_n.
\end{equation}
These supply tight-binding models of a quantum particle subjected to an external potential energy, $V$. Broadly speaking, one is interested in what features of a physical model can be extracted solely from the large-scale \emph{structure} of said model, rather than specific local features.
In the current context, one often studies \emph{ergodic models} in which the potential $(V_n)_{n \in \Z}$ is generated by a dynamical system, and thus the question of interest can be phrased as asking whether similar phenomena hold for ergodic families regardless of the exact method by which they are generated from the dynamical system in question.
For instance, the \emph{Anderson model}, in which $(V_n)_{n \in \Z}$ is a sequence of independent identically distributed random variables exhibits spectral localization (pure point spectrum with exponentially decaying eigenfunctions) for any choice of randomness in 1D; compare \cite{CarKleMar1987CMP, KunSou1980CMP}.
The \emph{trimmed Anderson model} \cite{ElgartKlein2014JST, ElgartSodin2017JST} introduces randomness only on a sub-lattice of sites and still exhibits localization; more generally,
for \emph{random word models}, localization is essentially universal: any ``nontrivial'' word model (that is any word model whose realizations are not simply periodic sequences) exhibits spectral localization and dynamical localization \cite{DamSimSto2004JFA, Rangamani2022AHP}.
In fact, the randomness does not even need to be identical at every site. Under relatively mild restrictions, localization holds for potentials given by independent random variables \cite{GK2026CAMS}.
In a similar vein, word models defined by suitable aperiodic subshifts tend to exhibit the same features that are observed in the single-site model: singular continuous zero-measure Cantor spectrum \cite{DamFilGohlke2022JST}.
It is of interest to see whether any other models exhibit such universal features.

Building on work of Avila \cite{Avila2009CMP} and some later simplifications in the regime of vanishing Lyapunov exponents \cite{DFL2017JST}, one such result for limit-periodic operators was shown in \cite{FGH2026JFA}, up to a spectral gap condition: if $X \subseteq \R^n$ is gap-rich, then a generic limit-periodic function $\Z \to X$ produces a Schr\"odinger operator with purely singular continuous spectral type supported on a spectrum that is a Cantor set of zero Hausdorff dimension.
We say that $X$ is gap-rich if for any energy away from a discrete set, any periodic operator whose potential is a concatenation of elements of $X$ can be perturbed to another periodic concatenation of elements of $X$ to open a spectral gap at the energy in question; see Definition~\ref{defn:schro_transfer_matrices} for the precise formulation.

Thus, up to the spectral gap condition described above, the generic features of limit-periodic operators observed in \cite{Avila2009CMP} are universal in the sense that any choice of ``building blocks'' produces the same generic spectral behavior.
As mentioned above, the proof relies on Avila's perturb-and-grow scheme from \cite{Avila2009CMP}, which has been implemented in various other scenarios such as  quantum walks \cite{FilOng2017JFA}.
For further background and results, we point the reader to the survey article \cite{DF2020LPSurvey}.

The results of \cite{FGH2026JFA} facilitate the identification of gap-rich sets among suitable algebraic subsets of $\R^n$ on which one can verify an explicit criterion.
While this is a significant advantage over previous works, there are approximately three nontrivial drawbacks, all of which we rectify in the current work:
\begin{itemize}
    \item The sets of admissible periodic potentials must be drawn from a ``nice'' set, where ``nice'' roughly means ``the image of a  vector-valued analytic function whose domain is a Banach space''. In particular, totally disconnected sets were out of reach of the ideas discussed in the previous work.
    \item While the ``non-ahyperbolicity'' criterion is simple in practice to check for each \emph{specific} instance, it still must be verified by hand for each example, usually by suitable \emph{ad hoc} techniques that do not readily generalize to other classes.
    Thus, the previous work provided no \emph{systematic} way to apply the techniques to broader classes of examples.
    Indeed, the reader can see that the examples worked out earlier each require simple (nevertheless distinct) computations in order to complete the desired analysis.
    \item The computations and ideas of \cite{FGH2026JFA} could mostly only be applied fruitfully to sets of the form $X = \prod K_j$ with each $K_j$ a closed subset of $\R^1$.
    For sets in which the potentials are ``correlated'' (e.g.\ if $K$ is an affine subspace of $\R^k$), the direct computations needed to verify the relevant inputs can very quickly become intractable. 
    An instructive example is supplied by the line  $x_1/a_1 = \cdots = x_k/a_k$ in $\R^k$, where each $a_k \neq 0$. If every $a_i = 1$, then the relevant monodromies are simply the $k$th powers of the basic transfer matrices, and there is no real obstacle.
    However, as soon as some $a_i$ is not $1$ and $k$ is not very small, then the computations of the previous work become difficult to implement.
\end{itemize}

In this work, we broaden the scope of previous results substantially by combining existing analysis with some tools from algebra in order to overcome the  obstacles enumerated above.
To formulate the theorem, we recall that the Zariski topology on $\R^n$ refers to the weakest topology on $\R^n$ for which all polynomial functions $\R^n \to \R$ are continuous, and we say that a set $K \subseteq \R^n$ is \emph{locally Zariski-dense} (LZD) if   $K \cap B(x,\eps)$ is dense in $\R^n$ in the Zariski topology  for every $x \in K$ and $\eps>0$, where $B(x,\eps)$ denotes the ball of radius $\eps$ centered at $x$ in the Euclidean metric. 
We denote by $\polynomial(\R^n, \R^m) \subseteq C(\R^n, \R^m)$ the set of all polynomial functions $\R^n \to \R^m$, that is, those functions $p:\R^n \to \R^m$ where $p = (p_1,\ldots,p_m)$ and each $p_i : \R^n \rightarrow \R$ is a polynomial function with real coefficients. For the $m=1$ case we shall write\footnote{Here we make a minor point that we view polynomials in $\polynomial(\R^n)$ as \emph{functions}, that is, as elements of $C(\R^n)$, not as abstract algebraic objects, which is why we eschew the traditional algebraic notation $\R[x_1,\ldots,x_n]$.}  $\polynomial(\R^n, \R) = \polynomial(\R^n)$ and $C(\R^n;\R) = C(\R^n)$.

\begin{theorem}\label{t:main}
    Let $n,m \in \N$ and $\ell \in \N_0 = \N \cup \{0\}$. If $K \subseteq \R^n$ is LZD,  $p \in \polynomial(\R^n,\R^m)$ is nonconstant, and $F \subseteq \R^\ell$ is finite, then ${p[K]} \times F \subseteq \R^{m+\ell}$ is gap-rich.
\end{theorem}

As a consequence, we can completely characterize the gap-rich subsets in dimension one.

\begin{corollary}\label{t:one_dimensional_case:alt}
    Let $X \subseteq \R$ be given. The following are equivalent:
    \begin{enumerate}[label={\rm(\alph*)},itemsep=3pt]
        \item \label{item:1Dalt:crowd} $X$ is crowded (i.e., $X$ has no isolated points)
        \item \label{item:1Dalt:lzd} $X$ is LZD
        \item \label{item:1Dalt:gaprich} $X$ is gap-rich
    \end{enumerate}
\end{corollary}

\begin{remark}\label{rmk:maintheorem}
\mbox{\,}
\begin{enumerate}[label={\rm(\arabic*)},itemsep=3pt]
    \item The setting $\ell = 0$ in the theorem should be interpreted as the setting in which $F$ is absent; hence the conclusion holds for sets of the form $p[K]$ with $K$ locally  dense in the Zariski topology and $p$ a nonconstant polynomial.

    \item Of course, $\R^m$ is locally dense in the Zariski topology on itself, so as a consequence, we see that any set that is the image of a nonconstant polynomial function is gap-rich. As mentioned above, this result by itself was already beyond the reach of \cite{FGH2026JFA}.
    
    \item The exact order of the factors above is not really important, and indeed one can  see from the proof that one can freely intersperse the coordinates corresponding to $p[K]$ and to $F$.
    More precisely, if we call any set of the form $p[K]$ a \emph{good} set if $K$ is locally dense in the Zariski topology and $p$ is a nonconstant polynomial, then the arguments herein show that any Cartesian product of good and finite sets is gap-rich. \qedhere
\end{enumerate}
\end{remark}

One can also handle the case of general analytic functions, but here the set of domains to which one can apply the result needs to be diminished somewhat to those having dense interior.\footnote{Both \textit{dense} and \textit{interior} in this statement refer to the Euclidean topology, which shall be the default topology in this article, see also Remark~\ref{rmk:euclidean_vs_zariski}.}

\begin{theorem} \label{t:analytic}
    Let $n, m \in \N$. If $K \subseteq \R^n$ has dense interior {\rm(}i.e., $\overline{K^\circ}\supseteq K${\rm)} and $f:\R^n \to \R^m$ is a nonconstant analytic function, then ${f[K]}$ is gap-rich.
\end{theorem}

In the remainder of this note, we set out definitions and important notions in Section~\ref{sec:key_ingredient}, establish some algebraic preliminaries in Section~\ref{sec:algebra}, and prove the main result in Section~\ref{sec:proof}. Section~\ref{sec:examples} collects some representative examples, and Section~\ref{sec:concluding_remarks} summarizes the status of the results so far, and lists a few open problems.

\section{Setting and Definitions}
\label{sec:key_ingredient}

We begin by collecting some of the relevant definitions and notations.

\begin{definition}
    A matrix $A \in \sltwor$ is called \textit{elliptic} if $|\trace{A}|<2$, \textit{hyperbolic} if $|\trace{A}|>2$, and \textit{parabolic} if $|\trace{A}|=2$ and $A \neq \pm I$.
\end{definition}

\begin{definition}
    [Word model] \label{defn:words_and_concatenation}
Given a metric space $(Y,d)$ and $k \in \N$, we denote a typical element in the product space $Y^k$ by $\vect y = y_1\cdots y_k$. We refer to $\vect y$ as a \textit{word}. The \textit{concatenation} of words $\vect y \in Y^k$ and $\vect z \in Y^\ell$ is denoted by
    \[
        \vect y \sharp \vect z= y_1 \cdots y_k z_1 \cdots z_\ell \in Y^{k + \ell}.
    \]
    We will sometimes shorten this further and simply write $\vect y \vect z $ in place of $\vect y \sharp \vect z $. We denote by $\vect y^{\sharp m}$ the result of concatenating $m$ copies of $\vect y$. The \textit{set of words over $Y$} is denoted by
    \[
        Y^\star := \bigcup_{k=1}^\infty Y^k.
    \]
    For $\vect y \in Y^k$, $\vect z \in Y^\ell$, we define 
    \begin{equation}
        d(\vect y, \vect z)
        := \sup_{1\leq j \leq  k \ell} d\left((\vect y^{\sharp \ell})_j , (\vect z^{\sharp k})_j \right).
    \end{equation}
    In particular, if $\vect y$ and $\vect z$ both belong to $Y^k$ for some $k \in \N$, then $d(\vect y, \vect z)$ simply denotes the distance between them in the uniform metric on the product space $Y^k$.
\end{definition}

We will extend functions on $Y$ to functions on $Y^\star$ as follows:
\begin{definition}\label{defn:extend_functions}

If $f:Y \to \R^m$ and $\vect y = y_1 \cdots y_k \in Y^k$, then
$f$ extends uniquely to a map $Y^\star \to (\R^m)^\star$ (which we denote by the same letter) satisfying
\[
    f(\vect x \sharp \vect y) = f(\vect x) \sharp f(\vect y).
\]
    Given a function $A:Y \times \R \rightarrow \sltwor$, there is a unique extension $A:Y^\star \times \R \to \sltwor$  such that
    \begin{equation} \label{eq:extend_functions}
            A(\vect y \sharp 
            \vect z, E) = A(\vect z,E)A(\vect y, E).
    \end{equation}
    Similarly, $B: Y \to \sltwor$, can be extended to $B:Y^\star \to \sltwor$ by $B(\vect y 
    \sharp \vect z) = B(\vect z) B(\vect y)$.
\end{definition}

\begin{definition}[Topological notions]
Let $(Y,d)$ denote a metric space and $X \subseteq Y$.
The \emph{interior} (respectively \emph{closure}) of $X$ will be denoted by $X^\circ$ (respectively, $\overline{X}$).
We say that $X$ is \emph{crowded} if it has no isolated points and \emph{perfect} if it is both closed and crowded.
We say that $X$ has \emph{dense interior} if $\overline{X^\circ} \supseteq X$.
\end{definition}

For later use, we record two straightforward observations concerning sets with dense interior.

\begin{lemma}\label{lem:dense_interior}
    Suppose that $K\subseteq \R^k$ has dense interior. 
    Then $(K \cap B(x, \eps))^\circ$ is non-empty for every $x \in K$ and $\eps>0$.
    In particular, $K$ is crowded.
\end{lemma}

\begin{proof}
    Fix $x \in K$ and $\eps>0$ and note first that $(K \cap B(x, \eps))^\circ = K^\circ \cap B(x,\eps)$. Since $K \subseteq \overline{K^\circ}$, there exists $y \in K^\circ$ with $d(x,y)<\eps$, that is, $y \in K^\circ \cap B(x,\eps)$.
\end{proof}

\begin{lemma}\label{lem:dense_interior_union}
    If $K_1, K_2 \subseteq \R^k$ has dense interior, then so does $K_1\cup K_2$.
\end{lemma}

\begin{proof}
    Taking closure on both sides of $(K_1)^\circ \cup (K_2)^\circ \subseteq (K_1\cup K_2)^\circ$, we obtain $\overline{(K_1)^\circ} \cup \overline{(K_2)^\circ} \subseteq \overline{(K_1\cup K_2)^\circ}$. Since $K_1 \cup K_2 \subseteq \overline{(K_1)^\circ} \cup \overline{(K_2)^\circ}$ by assumption, the conclusion follows. 
\end{proof}

The main notion from \cite{FGH2026JFA} is the following:

\begin{definition}
    [Gap-rich maps and sets] \label{defn:gap_rich}
    Let $Y  $ be a  metric space, and $A : Y\times \R \rightarrow \sltwor$ be a continuous map.
    We say that $A$ is \textit{gap-rich} with exceptional set $S \subseteq \R$ if $S$ is discrete, and, for all  $E \in \R \setminus S$,  $\vect x \in Y^\star$, and  $\eps > 0$, there exist $\vect y = \vect y(E,\vect x, \varepsilon) \in Y^\star$ such that 
    \begin{enumerate}[label=(\roman*)]
        \item \label{item:GRdef:close} $d(\vect x, \vect y) < \eps$, and
        \item \label{item:GRdef:hyp}  $A(\vect y, E)$ is hyperbolic.
    \end{enumerate}
    When the exact form of $S$ is unimportant, we shall omit any mention of $S$ and simply say that $A$ is gap-rich.
    We say that $X \subseteq Y$ is gap-rich for $A$ if the restriction $A|_{X \times \R}$ is gap-rich.
\end{definition}

Below we record some basic properties concerning gap-rich sets.

\begin{proposition}\label{prop:basic_properties_gaprich}
Let $Y$ be a metric space and $A : Y \times \R \to \sltwor$ be continuous.
\begin{enumerate}[label=(\alph*)]
    \item \label{item:grIsolatedPoints} If $A$ is gap-rich with exceptional set $S$, then for any isolated point $x_0$ of $Y$, $ A(x_0,E)$ is hyperbolic for all $E \in \R \setminus S$. 
    
    \item \label{item:grClosure} $X \subseteq Y$ is gap-rich with exceptional set $S$ if and only if $\overline{X}$ is.
    \item \label{item:grUnion} If $X_1 \subseteq X_2 \subseteq \cdots$ are gap-rich with exceptional set $S$, then so is $\bigcup_{i=1}^\infty X_i$.
\end{enumerate}
\end{proposition}
\begin{proof}
\ref{item:grIsolatedPoints}
Assume on the contrary that $A$ is gap-rich with exceptional set $S$ and that $x_0$ is an isolated point of $Y$ such that $A(x_0,E)$ is not hyperbolic for some $E \in \R \setminus S$.
Choosing $\varepsilon>0$ so that $B(x_0,\varepsilon) = \{x_0\}$ in $Y$, it follows that the only  $\vect y \in Y^\ell$ satisfying $d(x_0,\vect y) < \varepsilon$  is $x_0^{\sharp \ell}$.
Furthermore, $A(x_0^{\sharp \ell},E) = A(x_0,E)^\ell$ is a power of a non-hyperbolic matrix and hence is itself non-hyperbolic for all $\ell$, so there is no $\vect y \in Y^\star$ with $d(x_0,\vect y) < \varepsilon$ for which $A(\vect y,E)$  is hyperbolic, which is a contradiction.
\medskip

\ref{item:grClosure} Assume first that $X$ is gap-rich with exceptional set $S$. Given $ \vect y \in \overline{X}^\star$, $E \in \R \setminus S$, and $\varepsilon>0$, first choose $\vect z \in X^\star$ with $d(\vect y , \vect z)< \varepsilon/2$.
 Since $X$ is gap-rich, we may choose $\vect w \in X^\star \subseteq \overline{X}^\star$ with $d(\vect z, \vect w)< \varepsilon/2$ for which $A(\vect w,E)$ is hyperbolic.
 Since $d(\vect y, \vect w) < \varepsilon$, this concludes the argument in this direction.
 \smallskip

Conversely, if $\overline{X}$ is gap-rich (with exceptional set $S$) and $\vect x  \in X^\star \subseteq \overline{X}^\star$, $E \in \R \setminus S$, and $\varepsilon>0$ are given, we can find $\vect y \in \overline{X}^n \subseteq \overline{X}^\star$ such that $d(\vect x, \vect y) < \varepsilon/2$ and $A(\vect y,E)$ is hyperbolic.
By the continuity assumption of $A(\cdot,E)$ and of the trace, we can then choose $\vect z \in X^n$ close enough to $\vect y$ to ensure both $d(\vect x, \vect z)< \varepsilon$ and $|\trace A(\vect z, E)| > 2$,  which finishes this direction.
 \medskip
 
\ref{item:grUnion} If $\vect x = x_1 \cdots x_m \in \left[\bigcup_{i=1}^\infty X_i \right]^\star$, then there is some $N$ such that $x_i \in X_N$ for all $i$, so we can apply gap-richness of $X_N$ to conclude. 
\end{proof}

It is important to note that we are seeking a word $\vect y$ in $Y^\star$. That is, we must perturb $\vect x$ \textit{within} the set of admissible words $Y^\star = \bigcup_{k=1}^\infty Y^k$. Hence the choice of the set $Y$ plays a critical role in the discussion.

The main example of $A : Y\times \R \rightarrow \sltwor$ that we have in mind is the transfer matrix of one-dimensional discrete Schr\"odinger operators \cite{DF2022ESO1}: 

\begin{definition}
    [Discrete 1D Schr\"odinger transfer matrices] \label{defn:schro_transfer_matrices}
    The \textit{transfer matrix} $T:\R \times \R \rightarrow \sltwor$ for the Schr\"odinger operator on $\ell^2(\Z)$, is given by 
    \begin{equation}\label{eqn:schro_transfer_matrices}
        T(x, E) = \begin{bmatrix}
            E-x & -1 \\
            1 & 0
        \end{bmatrix}. \qedhere
    \end{equation}
    In accordance with Definition~\ref{defn:extend_functions}, we also denote by $T$ the unique extension of $T$ to $\R^\star \times \R$ satisfying \eqref{eq:extend_functions}.
    As above, we then simply say that a set $X \subseteq \R^k$ is gap-rich if $T|_{X \times \R}$ is gap-rich in the sense of Definition~\ref{defn:gap_rich}.
\end{definition}

We shall now consider various choices of $Y \subseteq \R^k$, and ask whether $Y$ is gap-rich for the Schr\"odinger transfer matrix \eqref{eqn:schro_transfer_matrices}, as a map $T : Y \times \R \rightarrow \sltwor$.
This is indeed the main motivation for the discussion at hand as well as for the specific terminology.
More specifically, if $H_{\vect x} = \Delta + V_{\vect x}$ denotes the Schr\"odinger operator with periodic potential obtained by repeating the string $\vectx \in \R^k$, then  by Floquet theory (e.g.\ \cite[Section~7.2]{DF2024ESO2})
\begin{equation}\label{eqn:spec_vs_trace_periodic}
   \R \setminus \spectrum H_\vectx = \{E \in \R : |\trace T(\vectx,E)| > 2 \}.
\end{equation}
Thus, gap-richness of $T$ for a suitable class of blocks exactly corresponds to the ability to perturb (within the given allowable blocks) to open spectral gaps.

We conclude this section with a lemma showing that, in case of Schr\"odinger transfer matrices, gap-richness of $X$ forces $X$ to be crowded.

\begin{lemma}
    \label{lem:gr_implies_perfect}
    If $X \subseteq \R^k$ is gap-rich, then $X$ is crowded.
\end{lemma}

\begin{proof}
    Arguing by contrapositive, assume that $X$ has an isolated point $\vect{x}_*$.
    By standard results (cf.~\cite[Theorem~7.2.8]{DF2024ESO2}), $\trace T(\vect x_*, E)$ is a nonconstant polynomial in $E$ with all zeros real and simple.
    In particular, the set of $E$ for which $T(\vect x_*,E)$ is non-hyperbolic is non-discrete, so $X$ is not gap-rich by Proposition~\ref{prop:basic_properties_gaprich}.\ref{item:grIsolatedPoints}.
\end{proof}

\section{Algebraic Preliminaries} \label{sec:algebra}


As indicated in Theorem~\ref{t:main}, the geometric properties of the set $K\subseteq \R^k$ play a key role in the analysis that follows. To prepare for this, we recall in this section some notions from algebraic geometry and collect some useful results. 
For the reader's convenience, we present this algebraic background in a relatively detailed fashion, with references for important facts. In this section, we shall fix $k \in \N$.

Given a set $S \subseteq \polynomial{(\R^k)}$ of polynomials, the \emph{variety} generated by $S$ is given by
\begin{equation} \label{eq:varietydef}
\Var (S):=
\{x \in \R^k : f(x) = 0 \ \forall f \in S\}.
\end{equation} 
The \emph{Zariski topology} on $\R^k$ is the unique topology on $\R^k$ for which every closed set is of the form $\Var(S)$ for some $S \subseteq \polynomial{(\R^k)}$.
Equivalently, the Zariski topology is the coarsest topology for which all polynomial functions $\R^k \to \R$ are continuous.
A set $K \subseteq \R^k$ is called \emph{Zariski-dense} if $K$ is dense in the Zariski topology.

\begin{definition}
    \label{defn:lzd}
    We say that $K \subseteq \R^k$ is \emph{locally Zariski-dense} (LZD) if $K$ is nonempty and $K \cap B(x,\eps)$ is Zariski-dense for every $x \in K$ and every $\eps>0$.
\end{definition}

\begin{remark}
    \label{rmk:euclidean_vs_zariski}
    Here and henceforth, we shall take care to distinguish between the Euclidean and Zariski topologies on $\R^k$. We adopt the convention that the Euclidean topology is to be taken as default, whereas any use of the Zariski topology will be explicitly stated. For instance, the closure $\overline{K}$ of a set $K\subseteq \R^k$ refers to the Euclidean-closure, $B(x,\eps)$ refers to the Euclidean-open ball at $x \in \R^k$ with radius $\eps>0$, and ``locally" in LZD refers to taking a Euclidean-open neighborhood of a point $x \in K$. 
\end{remark}

We record a simple but useful observation about Zariski-dense sets that will be used throughout the article. This is well-known, but we include the short proof for the reader's convenience.

\begin{proposition}\label{prop:zariskidense_and_polydet}
    A set $K \subseteq \R^k$ is Zariski-dense if and only if it is has the property that
    \begin{align}\label{eqn:pd_property}
        \text{for every $p \in \polynomial{(\R^k)}$, $p|_K \equiv 0$ implies that $p \equiv 0$.}
    \end{align}
\end{proposition}

\begin{proof}
    $(\Rightarrow)$
    Arguing by contraposition, assume that $K \subseteq \R^k$ does not satisfy \eqref{eqn:pd_property}.
    We may then choose
    \begin{align}\label{eqn:nonzero_polys_vanishing_on_k}
        p^\ast \in S_K := \{ p \in \polynomial{(\R^k)} \,:\, \text{$p$ nonconstant, $p(a)=0$ for all $a \in K$} \}
    \end{align}
    and note that $K \subseteq \Var(S_K) \subseteq \Var(\{ p^\ast \})$. 
    Since $p^\ast$ is a nonconstant polynomial, $\Var(\{p^\ast\}) \neq \R^k$ and $\Var(\{p^\ast\})$ is closed in the Zariski topology. Thus, $K$ is not Zariski-dense.  
    
    $(\Leftarrow)$
    Suppose that $K \subseteq \R^k$ satisfies \eqref{eqn:pd_property}. Denote by $\mathbf{I}(K)$ the set of all polynomials that vanish on $K$ (the ideal of $K$):
    \begin{align}\label{eqn:ideal_generated_by_k}
        \mathbf{I}(K) = \{ p \in \polynomial{(\R^k)} \,:\, p(a)=0 ~\text{for all $a \in K$} \}
        = S_K \cup \{ 0 \},
    \end{align}
    where $S_K$ is defined as in \eqref{eqn:nonzero_polys_vanishing_on_k}. Then, the Zariski-closure of $K$ is $\Var{(\mathbf{I}(K))}$ \cite[Chapter 4.4 Prop.\,1]{cox_little_oshea}, and $\mathbf{I}(K) = \{ 0 \}$ by  assumption \eqref{eqn:pd_property}. This implies that the Zariski-closure of $K$ is $\Var{(\{0 \})} = \R^k$. That is, $K$ is Zariski-dense.
\end{proof}

As a consequence of Proposition~\ref{prop:zariskidense_and_polydet}, we have a characterization of LZD sets

\begin{corollary}\label{cor:lzd_lpd_equivalence}
    A set $K\subseteq \R^k$ is LZD if and only if it has the property that
    \begin{align}\label{eqn:lpd_property}
        \text{for every $a \in K$, $\eps >0$, and $p \in \polynomial{(\R^k)}$, $p|_{K\cap B(a,\eps)} \equiv 0$ implies that $p \equiv 0$.}
    \end{align}
\end{corollary}

We now turn to the problem of identifying LZD sets.
In dimension one, we can deduce a complete description of LZD sets, for we recall that the Zariski topology in dimension one can be characterized by its cardinality: the Zariski-closed sets of $\R$ are the finite sets (including the empty set) and $\R$ itself \cite[Example~1.1.1]{hartshone}. As a consequence, the Zariski-dense sets of $\R$ are precisely the infinite sets. More precisely, we have the following:

\begin{lemma}\label{lem:one_dim_lzd_equiv_perfect}
If $K \subseteq \R^k$ is LZD, then it is crowded.
If $k=1$, then the converse holds as well.
\end{lemma}

\begin{proof}
    Arguing by contrapositive, assume $K \subseteq \R^k$ is not crowded, that is, there exists an isolated point  $x^* \in K$.
    Then for some $\varepsilon>0$, $K \cap B(x^*,\varepsilon) = \{x^*\} = \Var(\{x_1 - x_1^\ast, \cdots, x_k - x_k^\ast \})$, 
    from which it follows that $K$ is not LZD.

    If $k=1$ and $K$ is crowded, then for any $x \in K$ and $\varepsilon>0$, $K \cap B(x,\varepsilon)$ is infinite (hence Zariski-dense), implying that $K$ is LZD.
\end{proof}

The following example may be useful for the reader to keep in mind:
\begin{example}\label{eg:nice_zd_example}
    For any $k \in \N$, $\Z^k$ is Zariski-dense in $\R^k$ (e.g.~\cite[Lemma~2.1]{alon_tarsi1992}) but not LZD, since it is not crowded.
\end{example}

We postpone further examples to Section~\ref{sec:examples}, after proving the main results in Section~\ref{sec:proof}. 

Below we show two ways of constructing \textit{bigger} LZD sets from existing ones, namely, by taking Cartesian products and unions. One should contrast the situation with gap-rich sets, as it is presently unknown if taking Cartesian products or unions of gap-rich sets preserve gap-richness (compare Question~\ref{quest:gr:operations} in Section~\ref{sec:concluding_remarks}).

\begin{lemma} \label{lem:prodLZD}
    If $k_j \in \N$ and $K_j \subseteq \R^{k_j}$ is LZD for each $j=1,2,\ldots,n$, then  $K=K_1 \times \cdots \times K_n$ is LZD in $\R^{k_1+\cdots + k_n}$.
\end{lemma}

\begin{proof}
    By induction, it suffices to consider the case $n=2$.
    To that end, assume that $K_1$ and $K_2$ are LZD, and fix $(x^\ast, y^\ast) \in K_1 \times K_2 =: K$ and $\eps>0$.
    Suppose $p \in \polynomial{(\R^{k_1+k_2})}$ is a polynomial satisfying
    \begin{align}
    \begin{split}
        \label{eqn:lpd_cartesian_pdt}
        p(\widetilde{x}, \widetilde{y}) = 0 \quad \text{for all} \quad (\widetilde{x}, \widetilde{y}) \in ~ 
        &K \cap \left( B(x^\ast, \eps) \times B(y^\ast, \eps)\right) \\
        &\quad= \left( K_1 \cap B(x^\ast, \eps) \right) \times \left( K_2 \cap B(y^\ast, \eps) \right).
    \end{split}
    \end{align}

    Let us fix $\widetilde{x} \in K_1 \cap B(x^\ast, \eps)$ and consider the polynomial $q \in \polynomial{(\R^{k_2})}$ defined by $q(\widetilde{y}) = p(\widetilde{x}, \widetilde{y})$. By \eqref{eqn:lpd_cartesian_pdt}, $q$ vanishes on $K_2 \cap B(y^\ast,\eps)$. Thus the LZD property of $K_2$ implies that $q(\widetilde{y}) = 0$ for all $\widetilde{y} \in \R^{k_2}$. That is,
    \begin{align}\label{eqn:lpd_cartesian_pdt_v2}
        p(\widetilde{x}, \widetilde{y}) = 0 \quad \text{for all} \quad \widetilde{x} \in K_1 \cap B(x^\ast, \eps) \quad \text{and} \quad \widetilde{y} \in \R^{k_2}.
    \end{align}
    
    Reversing the roles of $\widetilde{x}$ and $\widetilde{y}$ in the above paragraph, and applying the LZD property of $K_1$ now gives us for each fixed $\widetilde{y} \in \R^{k_2}$, that $p(\widetilde{x}, \widetilde{y}) = 0$ for all $\widetilde{x} \in \R^{k_1}$. That is, $p$ vanishes on the whole of $\R^{k_1 + k_2}$, as required.
\end{proof}

\begin{lemma}\label{lem:unionLZD}
    If $k \in \N$ and $K_j \subseteq \R^k$ is LZD for each $j=1,2,\ldots,n$, then $K = K_1 \cup \cdots \cup K_n$ is LZD in $\R^k$.
\end{lemma}

\begin{proof}
    Again, it suffices to consider $n=2$.
    Fix $a \in K$, $\eps>0$, and a polynomial $p \in \polynomial{(\R^k)}$ with $p(x) = 0$ for all $x \in K\cap B(a, \eps)$. Assume without loss of generality, that $a \in K_1$. Since $K\cap B(a, \eps) = (K_1 \cap B(a,\eps)) \cup (K_2 \cap B(a,\eps))$, we obtain $p|_{K_1\cap B(a, \eps)} \equiv 0$. Now the LZD property of $K_1$ implies that $p$ must be the zero polynomial, as required.
\end{proof}

Lemma~\ref{lem:unionLZD} is the analogous statement to Lemma~\ref{lem:dense_interior_union} in the case of LZD sets.

Finally,  the following lemma, proved in \cite[Corollary 3.3]{EFGL2022JFA}, supplies one of the key inputs for the overall analysis.

\begin{lemma} 
    \label{lem:hyperbolic_magic}
    Suppose that $A,B \in \sltwor$ and let $G = \operatorname{SG}(A,B)$ denote the smallest semigroup containing $A$ and $B$. If $A$ and $B$ are both elliptic and do not commute, then $G$ contains a hyperbolic element.
\end{lemma}

\section{Proof of Main Results} \label{sec:proof}

In this section, we combine everything together to prove the main results.

\subsection{Technical Result}

We begin with a technical result underlying the proofs of Theorem~\ref{t:main} and~\ref{t:analytic}.

\begin{definition}

    Let $R$ be a unital subalgebra of $C(\R^n)$ (i.e.\ $R$ is a subset of $C(\R^n)$ containing the constant functions that is closed under pointwise sums and products) and $\scrK$ a nonempty collection of nonempty crowded subsets of $\R^n$.
    We say that $\scrK$ is \emph{locally $R$-determinative} if for every $f \in R$, $K \in \scrK$, $x \in K$, and $\delta>0$,
        \begin{equation*}
        \text{if $f|_{K  \cap B( x,\delta)} \equiv 0$, then $f \equiv 0$ on $\R^{n }$.} \qedhere
    \end{equation*}
\end{definition}

For a subalgebra $R$ as above, we denote by $R^{\otimes M}$ the $M$-fold tensor product of $R$ with itself, which can naturally be identified with a subspace of $C(\R^{nM})$ by sending elementary tensors to separable functions, that is,
\begin{equation}
    R^{\otimes M}  \ni f_1 \otimes \cdots \otimes f_M \mapsto f \in C(\R^{nM}), \quad f(x_1,\ldots, x_M) =\prod_{i=1}^M f_i(x_i). 
\end{equation}
Under this identification, $R^{\otimes M}$ consists of finite linear combinations of separable functions.

\begin{proposition}
    If $\scrK$ is locally $R$-determinative, then   for every $K \in \scrK$, $M \in \N$, $g \in R^{\otimes M}$, $\vect x \in K^M$ and $\delta > 0$,
    \begin{equation*}
        \text{if $g|_{K^M \cap B(\vect x,\delta)} \equiv 0$, then $g \equiv 0$ on $\R^{nM}$.} \qedhere
    \end{equation*}
\end{proposition}

\begin{proof}
    This follows by a repetition of the argument that proved Lemma~\ref{lem:prodLZD}.
\end{proof}

\begin{proposition}\label{prop:local_to_global} 
Let $R$ be a unital subalgebra of $C(\R^n)$, $\scrK$ be a nonempty collection of nonempty crowded subsets of $\R^n$, and $f:\R^n\rightarrow \R^m$ and $A:\R^m \to \sltwor$ be continuous functions such that 
    \begin{enumerate}[label=(\roman*)]
        \item \label{hyp:ring_member} The entries of $A\circ f$ belong to $R$.

        \item \label{hyp:local_to_global} $\scrK$ is locally $R$-determinative.

        \item \label{hyp:kickstart} There exists  $\vect w \in (f[\R^n])^\star$ for which $A(\vect w)$ is hyperbolic.
    \end{enumerate}
    Then, for every $K \in \scrK$, $\eps>0$, and $\vect y \in f[K]^\star$, there exists $\vect z \in f[K]^\star$ such that $d(\vect y,\vect z)<\eps$ and $A(\vect z)$ is hyperbolic.
\end{proposition}

A few words are in order regarding the statement of Proposition~\ref{prop:local_to_global}. First, one should think of $K$ and $f$ as fixed according to the statement of Theorems~\ref{t:main} and~\ref{t:analytic}. Second, one should think of $A:\R^m \rightarrow \sltwor$ as the transfer matrix at a \textit{fixed} energy $E$, and the goal of the proposition is to verify the conditions \ref{item:GRdef:close} and \ref{item:GRdef:hyp} of Definition~\ref{defn:gap_rich} at energy $E$.

Thus, the key hypotheses lie in~\ref{hyp:local_to_global} and~\ref{hyp:kickstart}. In plain words,~\ref{hyp:local_to_global} is a requirement that the entries of the matrix $A\circ f$ satisfy a suitable ``local-to-global" property. More concretely, we intend to take $R$ as the ring $\polynomial{(\R^n)}$ and $\scrK$ as the collection of LZD sets in the polynomial case (Section~\ref{sec:poly_case}), where~\ref{hyp:local_to_global} is satisfied by Corollary~\ref{cor:lzd_lpd_equivalence}. Correspondingly, in Section~\ref{sec:analytic_case}, we shall take $R$ as the ring of real analytic functions $C^\omega(\R^n;\R) = C^\omega(\R^n)$ and $\scrK$ as the collection of sets having dense interior, where~\ref{hyp:local_to_global} is satisfied by the identity  theorem for real-analytic functions, which we will discuss in more detail later. As for~\ref{hyp:kickstart}, we remark that the word $\vect w$ lies in $(f[\R^n])^\star$, while $\vect z$ lies in $(f[K])^\star$. The existence of a word $\vect w$ satisfying~\ref{hyp:kickstart} shall be supplied using a separate argument.

Note that if $f$ were a constant function, say $f[\R^n] = f[K] = Y = \{ w \}$, then $Y^\star = \{ w^{\sharp L} : L \in \N\}$, and so $A(w)$ must be hyperbolic by condition~\ref{hyp:kickstart}.

\begin{proof}
    [Proof of Proposition~\ref{prop:local_to_global}]  
    Fix $K \in \scrK$,  $\eps>0$,  $L \in \N$, and $\vect y \in f[K]^L \subseteq (\R^m)^L$, and write
    \begin{align}
    \label{eqn:l2g_given_word}
        \vect y = y_1\cdots y_L
        \quad \text{and} \quad
        y_i = f(x_i),
        \quad \text{where} \quad
        x_i \in K \subseteq \R^n ~\text{and}~ 1\leq i\leq L.
    \end{align}
Denote $\vect x = x_1 \cdots x_L$ and, by continuity, pick $\delta > 0$ such that $f(B(x_i,\delta)) \subseteq B(y_i,\eps)$ for all $1\leq i \leq L$.

    By~\ref{hyp:kickstart}, there exists $N \in \N$ and $\vect w \in f[\R^n]^N \subseteq (\R^m)^N$ such that $A(\vect w)$ is hyperbolic. Put $M = \lcm{(L,N)}$, which gives
    %
    %
     $\vect y^{\sharp M/L}$, $\vect w^{\sharp M/N} \in f[\R^n]^M \subseteq (\R^m)^M$. 
    Define the matrix-valued map $F: (\R^n)^M \rightarrow \R^{2\times 2}$ by
    \begin{align}
    \begin{split}
        F(\widetilde{\vect x})
        &= F(\widetilde{x}_1, \cdots, \widetilde{x}_M)
        = [A(f(\widetilde{x}_1) \cdots f(\widetilde{x}_M)), A(\vect y^{\sharp M/L})] \\
        &= [A(f(\widetilde{\vect x})), A(\vect y^{\sharp M/L})]
        := A(f(\widetilde{\vect x}))A(\vect y^{\sharp M/L}) - A(\vect y^{\sharp M/L})A(f(\widetilde{\vect x})),
    \end{split}
    \end{align}
    and the map $G : (\R^n)^M \rightarrow \R$ with
    \begin{align}
        G(\widetilde{\vect x}) = \trace{A(f(\widetilde{\vect x}))}. 
    \end{align}

    We shall now split our discussion into three cases according to the value of $\trace{A(\vect y)}$.

    \textbf{\boldmath Case 1:  $|\trace{A(\vect y)}|>2$.} Then $A(\vect y)$ is by definition, hyperbolic. The conclusion of the proposition is trivially satisfied by taking $\vect z = \vect y$.
    \medskip

    \textbf{\boldmath Case 2: $|\trace{A(\vect y)}| < 2$.} Since $A(\vect y)$ is elliptic and $A(\vect w)$ is hyperbolic, we note that
    \begin{align}
        \label{eqn:l2g_kickstart_elliptic}
        A(\vect y^{\sharp M/L})
        =A(\vect y)^{M/L}~\text{is elliptic}
        \quad \text{and} \quad
        A(\vect w^{\sharp M/N}) 
        =A(\vect w)^{M/N}~\text{is hyperbolic}.
    \end{align}

    Then, with $\vect x$ as chosen in \eqref{eqn:l2g_given_word}, we have 
    \begin{align}
        F(\vect x^{\sharp M/L}) = [A(f(\vect x)^{\sharp M/L}), A(\vect y^{\sharp M/L})] = \begin{psmallmatrix}
            0 & 0 \\
            0 & 0
        \end{psmallmatrix}.
    \end{align}
    
    We claim that the restriction $F|_{K^M \cap B(\vect x^{\sharp M/L}, \delta)}$ cannot be identically zero (in $\R^{2\times 2}$). Indeed, suppose that were the case. Then~\ref{hyp:local_to_global} applied to each of the four entries of $F$ would imply that $F(\widetilde{\vect x}) = \begin{psmallmatrix}
        0 & 0 \\
        0 & 0
    \end{psmallmatrix}$ for all $\widetilde{\vect x} \in (\R^n)^M$. This is in contradiction with \eqref{eqn:l2g_kickstart_elliptic}, since $\vect y^{\sharp M/L}, \vect w^{\sharp M/N} \in f[\R^n]^M$, and elliptic and hyperbolic matrices do not commute. It is important to note that we need $\vect w^{\sharp M/N} \in f[\R^n]^M$ here, and not just that $\vect w^{\sharp M/N} \in (\R^m)^M$.

    There must therefore exist some $\vect x^\ast = x_1^\ast \cdots x_M^\ast \in K^M \cap B(\vect x^{\sharp M/L}, \delta)$ such that $F(\vect x^\ast) \neq \begin{psmallmatrix}
        0 & 0 \\
        0 & 0
    \end{psmallmatrix}$. Construct a new word $\vect y^\ast$ by applying $f$ to $\vect x^\ast$ entry-wise. That is, 
    \begin{align}\label{eqn:l2g_newword_elliptic}
        \vect y^\ast = y_1^\ast \cdots y_M^\ast = f(\vect x^\ast) = f(x_1^\ast) \cdots f(x_M^\ast) \in f[K]^M \cap B(\vect y^{\sharp M/L} ,\eps) \subseteq (\R^m)^M.
    \end{align}
    This gives us $[A(\vect y^\ast), A(\vect y^{\sharp M/L})] = F(\vect x^\ast) \neq \begin{psmallmatrix}
        0 & 0 \\
        0 & 0
    \end{psmallmatrix}$. Moreover, since $K$ is a crowded set, and ellipticity is an open condition on $G$, we may further choose $\vect x^\ast \in K^M \cap B(\vect x^{\sharp M/L}, \delta)$ such that $A(\vect y^\ast)$ is elliptic.

    Lemma~\ref{lem:hyperbolic_magic} now applies with $A = A(\vect y^{\sharp M/L})$ and $B = A(\vect y^\ast)$, to give a word $\vect z = \vect z(\vect y, \vect y^\ast) \in f[K]^{kM}$ for some $k \in \N$, in which $A(\vect z)$ is hyperbolic, and is a product of $A(\vect y^{\sharp M/L})$ and $A(\vect y^\ast)$. The latter implies that
    $d(\vect z, \vect y) = |\vect z - \vect y^{\sharp kM/L}|_\infty \leq |\vect y^\ast - \vect y^{\sharp M/L}|_\infty \leq |\vect y^\ast - \vect y^{\sharp M/L}|_2 < \eps$, where the final inequality is due to \eqref{eqn:l2g_newword_elliptic}.
\medskip

    \textbf{Case 3: \boldmath  $|\trace{A(\vect y)}| = 2$.} Then
    \begin{align}\label{eqn:l2g_kickstart_parabolic}
        |\trace{A(\vect y^{\sharp M/L})}| = 2
        \quad \text{and} \quad
        A(\vect w^{\sharp M/N}) ~\text{is hyperbolic}.
    \end{align}
    We claim that the restriction $G|_{K^M \cap B(\vect x^{\sharp M/L},\delta)}$ cannot be identically constant.
    Indeed, by~\ref{hyp:local_to_global},
    $G|_{K^M \cap B(\vect x^{\sharp M/L},\delta)}$ being constant would imply that the original map $G:(\R^n)^M \rightarrow \R$ is constant. This is in contradiction with \eqref{eqn:l2g_kickstart_parabolic} and $\vect y^{\sharp M/L}, \vect w^{\sharp M/N} \in f[\R^n]^M$.

    There must therefore exist some $\vect x^\ast = x_1^\ast \cdots x_M^\ast \in K^M \cap B(\vect x^{\sharp M/L}, \delta)$, giving $\vect y^\ast = f(\vect x^\ast) \in f[K]^M \cap B(\vect y^{\sharp M/L},\eps)$, such that $A(\vect y^\ast)$ is either (a) hyperbolic, in which we are done, or (b) elliptic, in which we may apply the arguments of the previous case to $A(\vect y^\ast)$ to obtain a word $\vect z \in f[K]^\star$ with $d(\vect z,\vect y)< \eps$ and $A(\vect z)$ hyperbolic. This completes the proof.
\end{proof}

\subsection{The Case of Polynomial Images}\label{sec:poly_case}

We now have all the pieces needed to establish our first main result, Theorem~\ref{t:main}.

The following lemma establishes the finiteness of the exceptional set of energies.

\begin{lemma}\label{lem:nonconstantTracexFinite}
        Let $p:\R^n\rightarrow \R^m$ and $F\subseteq \R^\ell$ be as in the statement of Theorem~\ref{t:main}. There exists a finite set $S \subseteq \R$ such that for every $E \in \R\setminus S$, $L \in \N$, and $\vect z = z_1 \cdots  z_L \in F^\star$, the map $\Delta_{E,\vect z, L} = \Delta : \R^{nL} \rightarrow \R$  given by
        \begin{align*}
            \Delta(\vect x) := \trace{T(p(x_1) \sharp z_1 \sharp \cdots \sharp p(x_L)\sharp z_L, E)}, 
            \quad \text{where} \quad
            \vect x = (x_1, \cdots, x_L) \in (\R^n)^L,
        \end{align*}
        is a nonconstant polynomial. In particular, there exists $\hat{\vect x} \in (\R^{n})^L$ with $|\Delta(\hat{\vect x})| > 2$.
\end{lemma}

\begin{proof}
    Writing $p = (p_j)_{1\leq j \leq m}$, suppose first that all polynomials $p_1,\cdots,p_m \in \polynomial{(\R^n)}$ are nonconstant. Then we claim that $S$ may be taken to be the union of the coordinates over every element of $F$. That is,
    \begin{align}\label{eqn:exceptional_set_poly_case1}
        S = \bigcup_{z \in F} \{ \pi_1(z), \cdots, \pi_\ell(z) \},
    \end{align}
    where $\pi_j$ denotes the projection of $z \in \R^\ell$ onto the $j$-th coordinate.
    
    Indeed, let $E \in \R\setminus S$, $L\in \N$, and $z_1,\cdots z_L \in F$ be given. Since the degree of the composition of two nonconstant polynomials equals the product of the individual degrees, Lemma~\ref{lem:trace_nonvanishing_v2} implies that the leading-order term for $\Delta$ is contained in the following expression
    \begin{align}\label{eqn:leading_term_trace}
        \left( \prod_{\substack{1\leq i \leq L \\ 1 \leq j \leq \ell}} (E-\pi_j(z_i)) \right) \prod_{\substack{1\leq i \leq L \\ 1 \leq j \leq m}} (E - p_j(x_i)).
    \end{align} 
    Furthermore, observe that the factors $(E - \pi_j(z_i))$ are all non-zero, since $E \notin S$. Since $p_i$'s are nonconstant, \eqref{eqn:leading_term_trace} is a nonconstant polynomial in $\vect x \in \R^{nL}$, and thus the same is true for $\Delta$.

    In the case where some $p_j$ is a constant polynomial, say $p_j \equiv c \in \R$, one may take $S$ as the union of the right-hand side of \eqref{eqn:exceptional_set_poly_case1} with $\{c\}$. Then the argument above proceeds as before, with $(E - p_j(x_1))\cdots (E-p_j(x_L)) = (E-c)^L$ viewed as a non-zero constant factor in \eqref{eqn:leading_term_trace}.
\end{proof}

A crucial point here is that the exceptional set produced in Lemma~\ref{lem:nonconstantTracexFinite} is independent of $L$; otherwise, we would have to take a union of a countable collection of discrete sets, which could potentially be non-discrete.

\begin{proof}
    [Proof of Theorem~\ref{t:main}]

    We shall show that $Y = p[K]\times F$ is gap-rich with exceptional set $S$, where $S$ is the finite set taken from Lemma~\ref{lem:nonconstantTracexFinite}.
    
    Fix  $E \in \R\setminus S$, $\eps>0$,  $L \in \N$, and  $\vect w \in Y^L \subseteq Y^\star$. Let us write
    \begin{align}\label{eqn:given_word_poly}
    \begin{split}
        &\vect w = p(x_1)z_1 \cdots p(x_L)z_L = y_1 z_1\cdots y_L z_L, \\
        &\qquad\qquad \text{where} \quad
        x_i \in K \subseteq \R^n, 
        z_i \in F,
        \quad\text{and}\quad
        1\leq i \leq L.
    \end{split}
    \end{align}
    
    Furthermore, let us write
    \begin{align}\label{eqn:given_word_breakdown_poly}
    \begin{split}
        \vect y = y_1\cdots y_L \in p[K]^L \subseteq (\R^m)^L, \quad \text{and} \quad
        \vect z = z_1\cdots z_L \in F^L \subseteq (\R^\ell)^L.
    \end{split}
    \end{align}

    Consider the map $A = A_{E,\vect z, L} : \R^{mL} \rightarrow \sltwor$ given by
    \begin{align}\label{eqn:mapA_polycase}
        A(\widetilde{\vect y}) 
        = A(\widetilde{y}_1\cdots \widetilde{y}_L)
        = T(\widetilde{y}_1 z_1 \cdots \widetilde{y}_L z_L, E),
    \end{align}
    which we shall extend to $A: (\R^{mL})^\star \rightarrow \sltwor$ in accordance with Definition~\ref{defn:extend_functions}.

    We shall now apply Proposition~\ref{prop:local_to_global} with $K$ and $f = p$ as given by the statement of Theorem~\ref{t:main}, $\scrK$ as the collection of LZD sets of $\R^n$, $A = A_{E,\vect z,L}$ as given by \eqref{eqn:mapA_polycase}, $p[K]^L = p[K^L]$ for  ``$f[K]$" of the proposition, and $R = \polynomial{(\R^n)}$.

    We verify the hypotheses of Proposition~\ref{prop:local_to_global}. That $K$ is crowded follows from Lemma \ref{lem:one_dim_lzd_equiv_perfect}. That~\ref{hyp:ring_member} holds is clear. 
    Condition~\ref{hyp:local_to_global} is satisfied by Corollary~\ref{cor:lzd_lpd_equivalence}, and
    condition~\ref{hyp:kickstart} follows from Lemma~\ref{lem:nonconstantTracexFinite}.

    For $\eps$ and $\vect y \in p[K]^L \subseteq (p[K]^L)^\star$ as chosen at the start of the proof, Proposition~\ref{prop:local_to_global} implies that there is some $\vect y^\ast \in (p[K]^L)^{L^\ast} \subseteq (p[K]^L)^\star$, $L^\ast \in \N$, where
    \begin{align*}
        d_{(p[K]^L)^\star}(\vect y, \vect y^\ast) < \eps, 
        \quad \text{and} \quad
        A(\vect y^\ast) ~\text{is hyperbolic.}
    \end{align*}
    Thus, there exist a word $\vect w^\ast \in (p[K]\times F)^\star = Y^\star$ that satisfies 
    \begin{align}
        A(\vect y^\ast) = T(\vect w^\ast,E),
    \end{align}
    and has the property that
    \begin{align*}
        d_{Y^\star}(\vect w, \vect w^\ast)
        = d_{(p[K]^L)^\star}(\vect y, \vect y^\ast) < \eps,
        \quad \text{and} \quad
        T(\vect w^\ast,E) ~\text{is hyperbolic.}
    \end{align*}
    This completes the proof.
\end{proof}

We can now supply the proof of Corollary~\ref{t:one_dimensional_case:alt}.

\begin{proof}[Proof of Corollary~\ref{t:one_dimensional_case:alt}] 
    ~\ref{item:1Dalt:crowd} $\implies$~\ref{item:1Dalt:lzd} follows from Lemma~\ref{lem:one_dim_lzd_equiv_perfect};
    ~\ref{item:1Dalt:lzd} $\implies$~\ref{item:1Dalt:gaprich}  follows from Theorem~\ref{t:main} with $\ell =0$, $n=m=1$, and $p:\R \to \R$ given by $p(x) = x$;
    and~\ref{item:1Dalt:gaprich} $\implies$~\ref{item:1Dalt:crowd} follows from Lemma~\ref{lem:gr_implies_perfect}.
\end{proof}

\subsection{The Case of Analytic Images} \label{sec:analytic_case}

We now turn to our second main result, Theorem~\ref{t:analytic}.
We need a version of the identity theorem for real-analytic multi-variable functions. This is well-known, but we could not find a reference in exactly this form, so we supply a precise statement and proof here.

\begin{theorem}\label{thm:identity_theorem}
    Suppose $\Omega \subseteq \R^n$ is a connected domain and $f,g:\Omega \to \R$ are real-analytic functions with the property that $f|_U  \equiv g|_U$ for some open set $U \subseteq \Omega$. Then $f \equiv g$.
\end{theorem}

\begin{proof}
    This follows from \cite[Chapter~2]{krantz_parks}.
    More precisely, considering $h = f-g$ and 
    \[B = \{x \in \Omega : h \text{ vanishes on a neighborhood of } x \},\]
    then it is clear from the assumptions that $B$ is nonempty and open.
    Furthermore, a point $x$ in $\overline{B} \cap \Omega$ can be approximated by a sequence $\{x_n\}$ of points belonging to $B$. At each point $x_n$, it is clear that $h$ and all its partial derivatives vanish.
    By continuity, $h$ and all its partials vanish at $x$, so $h$ vanishes in a neighborhood of $x$ by \cite[Remark~2.2.4]{krantz_parks}.
    Since $B$ is nonempty, closed, and open, and $\Omega$ is connected, it follows that $B = \Omega$, that is $f\equiv g$, as desired.
\end{proof}

\begin{corollary}\label{cor:identity_dense_int}
    Suppose that $K \subseteq \R^n$ has dense interior, $f,g \in C^\omega(\R^n)$, and $f \equiv g$ on $K\cap B(x,\delta)$ for some $x \in K$ and $\delta>0$. Then $f \equiv g$ on $\R^n$.
\end{corollary}

\begin{proof}
    By Lemma~\ref{lem:dense_interior}, there exist some $x' \in (K\cap B(x,\delta))^\circ$. Since $(K\cap B(x,\delta))^\circ$ is open, there exist some $\delta'> 0$ such that $B(x',\delta') \subseteq (K\cap B(x,\delta))^\circ$. As the latter is contained in $K\cap B(x,\delta)$, we infer that $f \equiv g$ on the open set $B(x',\delta')$ and now Theorem~\ref{thm:identity_theorem} applies.
\end{proof}

We are now ready to prove Theorem~\ref{t:analytic}. 

\begin{proof}[Proof of Theorem~\ref{t:analytic}]
    Since $f$ is nonconstant, we fix a pair of reference points $x_*,y_* \in \R^n$ such that $f(x_*) \neq f(y_*)$.
        Let $S$ denote the set of $E$ for which either (a) $T(f(x_*),E)$ is parabolic, (b) $T(f(y_*),E)$ is parabolic, or (c) $T(f(x_*),E)$ commutes with $T(f(y_*), E)$.
    
        \begin{claim}
            $S$ is a finite set.
        \end{claim}
        \begin{claimproof}
            Since $\trace T(f(x_*),E)$ and $\trace T(f(y_*),E )$ are nonconstant polynomial functions of $E$ (Lemma~\ref{lem:trace_nonvanishing_v2}), the set of $E$ for which either $T(f(x_*),E)$ or $T(f(y_*),E)$ is parabolic is finite. 
    
            It remains to show that the set of $E$'s satisfying (c) is also finite. Since $f(x_*) \neq f(y_*)$, it follows that
            \[f(x_*) \sharp f(y_*) \neq f(y_*) \sharp f(x_*),\]
            so \cite[Lemma~5.11.8]{DF2024ESO2} implies that there exist some $E_0 \in \C$ such that 
            \begin{align}
                T(f(x_*) \sharp f(y_*),E_0) \neq T( f(y_*) \sharp f(x_*),E_0), \label{eqn:analytic_case_commuting}
                \intertext{or equivalently,}
                P(E_0) := [T(f(x_*),E_0), T(f(y_*),E_0)] \neq 0.
            \end{align}
            Let $i,j \in \{1,2\}$ be such that $P(E_0)_{ij} \neq 0$. Then the polynomial $E \mapsto P(E)_{ij}$ can only have finitely many zeros in $\C$, and hence in $\R$. In particular, the same is true for the commutator $P$. This proves the claim.
        \end{claimproof}
    
    \begin{claim}\label{claim:nonahyp}
        For all $E \notin S$, there exists $N \in \N$ and $\vect x \in (\R^n)^N$ such that $T(f(\vect x), E)$ is hyperbolic. (We remind the reader the notation from Definition~\ref{defn:extend_functions}.)
    \end{claim}
    
    \begin{claimproof}
        If $T(f(x_*),E)$ or $T(f(y_*), E)$ is hyperbolic, we can choose $N=1$, $\vect x = x_*$ or $\vect x = y_*$ as appropriate. Otherwise, by definition of $S$, it remains to consider the case where $T(f(x_*) ,E)$ and $T(f(y_*),E)$ are elliptic and noncommuting, in which case the claim follows from Lemma~\ref{lem:hyperbolic_magic}.
    \end{claimproof}
    \medskip
    
    We shall now show that $Y = f[K]$ is gap-rich with exceptional set $S$. Fix energy $E \in \R\setminus S$, $\eps>0$, and word $\vect y \in Y^\star$. We would like to apply Proposition~\ref{prop:local_to_global}, with $K$ and $f$ as given by the statement of Theorem~\ref{t:analytic}, $\scrK$ as the collection of non-empty dense interior subsets of $\R^n$, $A = A_E : \R^m \rightarrow \sltwor$ defined by $A(\widetilde{\vect y}) = T(\widetilde{\vect y}, E)$, and $R = C^\omega(\R^n;\R) = C^\omega(\R^n)$. Extend $A$ to $A:(\R^m)^\star \rightarrow \sltwor$ in accordance with Definition~\ref{defn:extend_functions}.

    We now verify the conditions of Proposition~\ref{prop:local_to_global}. 
    That $K$ is crowded follows from Lemma~\ref{lem:dense_interior}.
    That~\ref{hyp:ring_member} holds is clear.
    Condition~\ref{hyp:local_to_global} is satisfied by Corollary~\ref{cor:identity_dense_int}, and
    condition~\ref{hyp:kickstart} is satisfied by Claim~\ref{claim:nonahyp}.

    Proposition~\ref{prop:local_to_global} now applies to give $\vect z \in Y^\star$ for which $d(\vect y, \vect z)<\eps$ and $A(\vect z) = T(\vect z,E)$ is hyperbolic. This completes the proof.
\end{proof}

\section{Examples of Gap-Rich Sets} \label{sec:examples}

First, we recall the two examples that were shown to be gap-rich in \cite[Theorem 2.11-2.12]{FGH2026JFA}:

\begin{example}
    For $m, \ell \in \N$, and $b \in \R^\ell$ the sets $\{ x^{\sharp m} : x \in \R \}$ and $\R^m \times \{b\}$ are gap-rich. Observe that $\{ x^{\sharp m} : x \in \R \} = p[\R]$, where $p:\R \rightarrow \R^m$ with $p(x) = (x,\cdots,x)$, and $\R^m = q[\R^m]$, where $q:\R^m \rightarrow \R^m$ with $q(x_1,\cdots,x_m) = (x_1,\cdots,x_m)$. Thus Theorem~\ref{t:main} generalizes these two examples.
\end{example}

We now give some examples that fall under Theorem~\ref{t:main}. Below we consider the Cantor-middle thirds set $K_{1/3}$ as a representative example of a crowded set.

\begin{example}
    Since $K_{1/3}$ is crowded subset of $\R$, it is gap-rich by Corollary \ref{t:one_dimensional_case:alt}.
\end{example}

\begin{example}
    $(K_{1/3})^2 \subseteq \R^2$ is LZD by Lemma \ref{lem:prodLZD}, and hence $(K_{1/3})^2 \cup \left( (K_{1/3})^2 + (5,5) \right)$ is LZD by Lemma \ref{lem:unionLZD}. Both sets are gap-rich by Theorem \ref{t:main}.
\end{example}

\begin{example}
    The set $K_{1/3} \times \{ 0 \} \subseteq \R^2$ is not LZD because it is contained in the Zariski-closed set $\R\times \{ 0 \} = \Var{(p)}$ where $p \in \polynomial{(\R^2)}$ is given by $p(x_1,x_2) = x_2$. It is however gap-rich by Theorem \ref{t:main}.
\end{example}

\begin{example}
    The graph of $p(x)=x^2$ restricted to $K_{1/3}$, namely $\{ (x, x^2) : x \in K_{1/3} \}$, is gap-rich. In a similar vein, the twisted cubic $\{ (x,x^2,x^3) : x \in \R \}$ is gap-rich.
\end{example}

\begin{example}
    Let $p:\R^2 \rightarrow \R^3$ given by $p(x) = p(x_1,x_2) = (p_1(x), p_2(x), p_3(x)) = (x_1x_2, 3, x_2^2 - 2x_1)$ and $F = [0,1]^2 \cap \Z^2$. Then $p{[\R^2]}\times F \subseteq \R^5$ is gap-rich. Moreover, per Remark \ref{rmk:maintheorem}(3), $p_1[\R^2] \times F \times p_3[\R^2] \subseteq \R^4$ is also gap-rich.
\end{example}

Next we give some examples that fall under Theorem \ref{t:analytic}. 
A \emph{Cantorval} $K \subseteq \R$ is a compact set with dense interior and  uncountably many connected components, none of which is isolated.
The term was coined by Mendes and Oliveira in the context of sums of Cantor sets with dimensions summing close to one \cite{MendesOliveira1994nonlin}; see also \cite{MorMun2003Nonlin} and \cite{BaaGorMaz2024LMP} for more background.

\begin{example}
    Let $K \subseteq \R$ denote a Cantorval; $K^2$ has dense interior, and hence $K^2 \cup \left( K^2 + (5,5) \right)$ has dense interior by Lemma \ref{lem:dense_interior_union}, and is therefore gap-rich by Theorem \ref{t:analytic}.
\end{example}

\begin{example}
   Denote $S^1 = f[I]$, where $I = [0,2\pi]$ and $f:\R \rightarrow \R^2$ is given by $f(x) = (\cos(x),\sin(x))$. 
   Theorem~\ref{t:analytic} implies that $S^1$ is gap-rich.
   Let us note in passing that the associated transfer matrix map 
   \begin{align*}
       A(x,E)
       &= T(\sin(x),E) T(\cos(x),E) \\
       &=
       \begin{bmatrix}
           (E-\cos x)(E-\sin x) - 1 & -(E-\sin x) \\ E- \cos x & - 1
       \end{bmatrix}
   \end{align*}
   does \emph{not} satisfy the non-ahyperbolicity condition of \cite{FGH2026JFA}. Indeed, since
   \begin{align*}
       \trace{A(x,E)}
       & = E^2 - (\cos x  + \sin x)E + \cos x \sin x - 2,
   \end{align*}
   we see that there is a small $\varepsilon>0$ such that for every $E \in [1,1+\varepsilon]$, $A(x,E)$ is non-hyperbolic for every $x \in [0,2\pi]$. 
   In particular, $A$ cannot be non-ahyperbolic with exceptional discrete set $S \subseteq \R$, as $S$ must contain $[1,1+\eps]$.
   
   Furthermore, choosing a Cantorval $K \subseteq [0,2\pi]$, we also see that $f[K]$ is also gap-rich, again from Theorem~\ref{t:analytic}.
\end{example}

\begin{example}
    The graph of $f(x)=e^x$, namely $\{ (x, e^x) \in \R^2 : x \in \R \}$, is gap-rich.
\end{example}

\begin{example}
    Let $f:\R^2 \rightarrow \R^2$ be given by $f(x) = f(x_1,x_2) = (e^{x_1}, \cos{(x_2+x_1)})$. Then $f[K^2] \subseteq \R^2$ is gap-rich.
\end{example}

\section{Concluding Remarks} \label{sec:concluding_remarks}
While the present work broadens the scope of previous results significantly, it cannot cover all possible cases.
We have seen that a gap-rich set must necessarily be crowded (i.e., without isolated points) and that $X$ is gap-rich if and only if $\overline{X}$ (its closure in the Euclidean topology) is gap-rich, see Lemma~\ref{lem:gr_implies_perfect} and Proposition~\ref{prop:basic_properties_gaprich}.
Thus, it suffices to consider perfect sets and  we are left with the following problem:
\begin{question}\label{q:who_is_gaprich}
    Which perfect subsets of $\R^k$ are gap-rich?
\end{question}

On one hand, due to the absence of obvious counterexamples and the large number of positive examples identified in this article, one may guess that every crowded subset $Y \subseteq \R^k$ is gap-rich. On the other hand, to claim that ``all crowded sets are gap-rich" is to assert that the only information of $Y$ that matters for gap-richness is its \textit{cardinality}. More precisely, it would mean that knowing whether or not $Y \cap B(y,\eps)$ is infinite for every $y \in Y$ and $\eps>0$ completely determines the gap-richness of $Y$. From this perspective, it would seem unlikely that such a strong claim could hold true. At this point of writing, we opt not to turn the question above into a conjecture in either direction.

Note however, that the case $k=1$ is an exception to the above discussion, since we have demonstrated through a link to algebraic geometry that all crowded subsets of $\R$ are gap-rich (Corollary~\ref{t:one_dimensional_case:alt}).

Second, beyond the properties enumerated in Proposition~\ref{prop:basic_properties_gaprich}, it seems that the structure of the collection of gap-rich subsets of $\R^k$ is delicate. A basic question about this collection is whether it is closed under various set operations: It is clear that it is not closed under intersections since the intersection of two crowded sets need not be crowded. The main interest however lies in set operations that construct bigger sets from existing ones. More precisely, we explicitly pose the following:

\begin{question} \label{quest:gr:operations}
    If $X_1, X_2 \subseteq \R^k$ and $X_3 \subseteq \R^\ell$ are gap-rich, are any of the following sets $X_1 \cup X_2$, $X_1 + X_2$, and $X_1 \times X_3$ gap-rich?
\end{question}

The results of this article answers Question 2 partially in the affirmative, see for instance Lemmas~\ref{lem:prodLZD} and~\ref{lem:unionLZD}, and Lemma~\ref{lem:dense_interior_union} for analytic case.

\appendix
\section{Leading Order Terms of the Schr\"odinger Discriminant}

In this section, we take $T$ to be the Schr\"odinger transfer matrix (Definition~\ref{defn:schro_transfer_matrices}) and write the leading-order terms for the discriminant. 
This is well-known to experts and worked out here to keep the manuscript more self-contained.

\begin{lemma}
    [Leading-order term of $\trace{T}$] \label{lem:trace_nonvanishing_v2}
    For any $E \in \R$ and $\vect x = x_1 \sharp \cdots \sharp x_n \in \R^n$, $n \in \N$, we have 
    \begin{align}
        \trace{T(\vect x, E)} = \prod_{i=1}^n (E-x_i) + p(E,x_1,\cdots,x_n),
    \end{align}
    where $p$, viewed as a polynomial in $x_1,\cdots,x_n$, has degree at most $n-1$.
\end{lemma}

\begin{proof}
    We shall prove a slightly stronger claim, namely, that for every $E \in \R$ and $n \in \N$, one has
    \begin{align}\label{eqn:transfer_matrix_overall_structure_v2}
        T(\vect x, E) = \begin{bmatrix}
            \prod_{i=1}^n (E-x_i) + p_{n-1}     & q_{n-1} \\
            r_{n-1}                             & s_{n-1}
        \end{bmatrix},
    \end{align}
    where $p_{n-1}$, $q_{n-1}$, $r_{n-1}$, and $s_{n-1}$ are polynomials of degree at most $n-1$ in the variables $x_1,\cdots, x_n$. We shall suppress the dependence on $E$, as it is treated as a fixed constant in this proof.

    We proceed by induction on $n$. The base case $n=1$ follows from Definition~\ref{defn:schro_transfer_matrices}. For the inductive step, let us assume that \eqref{eqn:transfer_matrix_overall_structure_v2} is valid up to some $n-1\geq 1$. Then, for some polynomials $p_{n-2}$, $q_{n-2}$, $r_{n-2}$, $s_{n-2}$ in $x_1,\cdots,x_{n-1}$ of degree at most $n-2$, we have
    \begin{alignat}{3}
        T(\vect x, E) 
        &= T(x_1\sharp\cdots \sharp x_n, E)  \nonumber\\
        &= T(x_1\sharp \cdots \sharp x_{n-1}, E) T(x_n, E) \nonumber\\
        &= \begin{bmatrix}
            \prod_{i=1}^{n-1} (E-x_i) + p_{n-2}     & q_{n-2} \\
            r_{n-2}                                 & s_{n-2}
        \end{bmatrix} \begin{bmatrix}
            E-x_n     &-1 \\
            1         &0
        \end{bmatrix} \qquad
        &&\parbox{10em}{By the inductive \\hypothesis.} \nonumber\\
        &= \begin{bmatrix}
            \prod_{i=1}^n (E-x_i) + p_{n-1}     & q_{n-1} \\
            r_{n-1}                             & s_{n-1}
        \end{bmatrix},
    \end{alignat}
    where $p_{n-1}$, $q_{n-1}$, $r_{n-1}$, $s_{n-1}$ are polynomials in $x_1,\cdots,x_n$ of degree at most $n-1$. Thus \eqref{eqn:transfer_matrix_overall_structure_v2} is verified, and the conclusion of the Lemma follows.
\end{proof}

\bibliographystyle{amsalpha}
\bibliography{bibilography}

\end{document}